\documentclass[oneside,english]{amsart}
\usepackage[T1]{fontenc}
\usepackage[latin9]{inputenc}
\usepackage{amsthm}
\usepackage{amssymb}
\usepackage{esint}

\makeatletter
\numberwithin{equation}{section}
\numberwithin{figure}{section}
\theoremstyle{plain}
\newtheorem{thm}{\protect\theoremname}
  \theoremstyle{plain}
  \newtheorem{lem}[thm]{\protect\lemmaname}
  \theoremstyle{plain}
  \newtheorem{cor}[thm]{\protect\corname}
   \theoremstyle{plain}
  \newtheorem{claim}[thm]{\protect\claimname}

\makeatother

\usepackage{babel}
  \providecommand{\lemmaname}{Lemma}
  \providecommand{\claimname}{Claim}
\providecommand{\theoremname}{Theorem}
\providecommand{\corname}{Corollary}

\begin{document}
\global\long\def\f{\mathcal{F}}
\global\long\def\A{\mathcal{A}}
\global\long\def\B{\mathcal{B}}
\global\long\def\l{\mathcal{L}}
\global\long\def\pn{\mathcal{P}\left(\left[n\right]\right)}
\global\long\def\g{\mathcal{G}}
\global\long\def\s{\mathcal{S}}
\global\long\def\m{\mathcal{M}}
\global\long\def\Bin{\mathrm{Bin}}
\global\long\def\mh{\mu_{\frac{1}{2}}}
\global\long\def\mp{\mu_{p}}
\global\long\def\j{\mathcal{J}}
\global\long\def\d{\mathcal{D}}
\global\long\def\Inf{\mathrm{Inf}}
\global\long\def\p{\mathcal{P}}
\global\long\def\mpo{\mu_{p_{0}}}
\global\long\def\fpp{f_{p}^{p_{0}}}
\global\long\def\h{\mathcal{H}}
\global\long\def\OR{\mathrm{OR}}

\title{On the union of intersecting families}
\author{David Ellis}
\address{School of Mathematical Sciences, Queen Mary, University of London, Mile End Road, London, E1 4NS, UK.}
\email{d.ellis@qmul.ac.uk}
\author{Noam Lifshitz}
\address{Department of Mathematics, Bar Ilan University, Ramat Gan, Israel.}
\email{noamlifshitz@gmail.com}
\date{19th August 2018}

\begin{abstract}
A family of sets is said to be {\em intersecting} if any two sets in the family have nonempty intersection. In 1973, Erd\H{o}s raised the problem of determining the maximum possible size of a union of $r$ different intersecting families of $k$-element subsets of an $n$-element set, for each triple of integers $(n,k,r)$. We make progress on this problem, proving that for any fixed integer $r \geq 2$ and for any $k \leq (\tfrac{1}{2}-o(1))n$, if $X$ is an $n$-element set, and $\f  = \f_1 \cup \f_2 \cup \ldots \cup \f_r$, where each $\f_i$ is an intersecting family of $k$-element subsets of $X$, then $|\f| \leq {n \choose k} - {n-r \choose k}$, with equality only if $\f = \{S \subset X:\ |S|=k,\ S \cap R \neq \emptyset\}$ for some $R \subset X$ with $|R|=r$. This is best possible up to the size of the $o(1)$ term, and improves a 1987 result of Frankl and F\"uredi, who obtained the same conclusion under the stronger hypothesis $k < (3-\sqrt{5})n/2$, in the case $r=2$. Our proof utilises an isoperimetric, influence-based method recently developed by Keller and the authors. 
\end{abstract}

\maketitle

\section{Introduction}

Let $[n]:=\{1,2,\ldots,n\}$, and let $\binom{[n]}{k} :=\{S \subset [n]:\ |S|=k\}$. If $X$ is a set, we let $\p(X)$ denote the power-set of $X$. A family $\f \subset \p([n])$ is said to be $1$-\emph{intersecting} (or just {\em intersecting}) if for
any $A,B \in \f$, we have $A \cap B \neq \emptyset$.

One of the best-known theorems in extremal combinatorics is the Erd\H{o}s-Ko-Rado
theorem~\cite{EKR}, which bounds the size of an intersecting subfamily of $\binom{[n]}{k}$.
\begin{thm}[Erd\H{o}s-Ko-Rado, 1961]
\label{thm:ekr}
Let $k,n \in \mathbb{N}$ with $k < n/2$. If $\f \subset \binom{[n]}{k}$ is intersecting, then $|\f| \leq \binom{n-1}{k-1}$. Equality holds only if $\f = \{S \in \binom{[n]}{k}:\ j \in S\}$ for some $j \in [n]$.
\end{thm}

In 1987, Frankl and F\"uredi \cite{ff} considered the problem, first raised by Erd\H{o}s \cite{erdos} in 1973, of determining the maximum possible size of a union of $r$ 1-intersecting subfamilies of $\binom{[n]}{k}$, for each triple of integers $(n,k,r)$. They proved the following.
\begin{thm}[Frankl, F\"uredi, 1986]
\label{thm:ff}
If $\f \subset \binom{[n]}{k}$ is a union of two intersecting families, and $n > \tfrac{1}{2}(3+\sqrt{5})k \approx 2.62 k$, then $|\f| \leq {n \choose k}-{n-2 \choose k}$. Equality holds only if $\f = \{S \in \binom{[n]}{k}:\ S \cap \{i,j\} \neq \emptyset\}$, for some distinct $i,j \in [n]$. 
\end{thm}
They give an example which shows that the upper bound in Theorem \ref{thm:ff} does not hold provided if $n_0 \leq n \leq 2k+c_0\sqrt{k
}$, where $n_0,c_0>0$ are absolute constants with $n_0$ sufficiently large and $c_0$ sufficiently small; this disproved a conjecture of Erd\H{o}s in \cite{erdos}.

In this paper, we prove the following strengthening and generalisation of Theorem \ref{thm:ff}.
\begin{thm}
\label{thm:main}For each integer $r \geq 2$, there exists a constant $C = C(r) \in \mathbb{N}$ such that
the following holds. Let $n\ge2k+Ck^{2/3}$, and let $\f \subset \binom{\left[n\right]}{k}$
be a union of at most $r$ $1$-intersecting families. Then $\left|\f\right|\le\binom{n}{k}-\binom{n-r}{k}$, and equality holds only if $\f = \{S \in \binom{[n]}{k}:\ S \cap R \neq \emptyset\}$ for some $R \in \binom{[n]}{r}$.
\end{thm}

We note that even in the case $r=2$, the conclusion of Theorem \ref{thm:main} was previously known to hold only in the case $n-2k \geq \Omega(k)$ (i.e., only in the case $k/n \leq 1/2-\Omega(1)$). For the first time, we prove it for $n-2k = o(k)$ (i.e., for $k/n \leq 1/2-o(1)$), for any fixed $r \geq 2$, though the correct rate of growth of the $o(k)$ term here remains open. We conjecture that the conclusion of Theorem \ref{thm:main} holds for $n \geq 2k+c \sqrt{k}$ for $c = c(r)$ sufficiently large; this would be best-possible up to the value of $c$, as evidenced by the aforementioned construction of  Frankl and F\"uredi. It would be of great interest to determine the extremal families for every triple of integers $(n,k,r)$. 

We remark that if $\f \subset \p([n])$ is a union of at most $r$ $1$-intersecting subfamilies of $\p([n])$, then $|\f| \leq 2^n-2^{n-r}$. This was first proved by Kleitman \cite{kleitman} and is an easy consequence of the FKG inequality (see Lemma \ref{lem:fkg}); it is sharp, as evidenced by taking $\mathcal{F} = \cup_{i=1}^{r} \{S \subset [n]:\ i \in S\}$. In fact, we will use this bound in our proof of Theorem \ref{thm:main}.

We remark also that the problem considered here is closely related to the well-known Erd\H{o}s matching conjecture. Recall that the {\em matching number} $m(\f)$ of a family $\f \subset \p([n])$ is defined to be the maximum integer $s$ such that $\f$ contains $s$ pairwise disjoint sets. The 1965 \emph{Erd\H{o}s matching conjecture}~\cite{Erdos65} asserts that if $n,k,s \in \mathbb{N}$ with $n \geq (s+1)k$ and $\f \subset \binom{\left[n\right]}{k}$ with $m(\f) \leq s$, then
$$|\f| \leq \max\left\{{n \choose k} - {n-s \choose k}, {k(s+1)-1 \choose k}\right\}.$$
This conjecture remains open. Erd\H{o}s himself proved the conjecture
for all $n$ sufficiently large depending on $k$ and $s$, i.e.\ for all $n \geq n_0(k,s)$. The bound on $n_0(k,s)$ was lowered in several works: Bollob\'as, Daykin and Erd\H{o}s~\cite{BDE76} showed
that $n_0(k, s) \leq 2sk^3$; Huang, Loh and Sudakov~\cite{HLS12} showed that $n_0(k,s) \leq 3sk^2$, and Frankl and F\"{u}redi (unpublished)
showed that $n_0(k, s) =O(ks^2)$. One of the most significant results on the problem to date is the following theorem of Frankl~\cite{frankl-emc}.
\begin{thm}[Frankl, 2013]\label{thm:frankl-matching}
Let $n,k,s \in \mathbb{N}$ such that $n \geq (2s+1)k-s$, and let $\f \subset \binom{\left[n\right]}{k}$ such that $m(\f) \leq s$. Then
$|\f| \leq {n \choose k} - {n-s \choose k}$.
Equality holds if and only if there exists $S \in \binom{\left[n\right]}{s}$ such that $\f = \{F \in \binom{\left[n\right]}{k}:\ F \cap S \neq \emptyset\}$.
\end{thm}
Frankl and Kupavskii \cite{fk} recently proved that $n_0(k,s) \leq \tfrac{5}{3}ks-\tfrac{2}{3}s$ for all $s \geq s_0$ (for some absolute constant $s_0$), strengthening Theorem \ref{thm:frankl-matching} for $s$ sufficiently large.

Clearly, if $\f \subset \binom{\left[n\right]}{k}$ is a union of at most $r$ $1$-intersecting families, then $m(\f) \leq r$, so Theorem \ref{thm:frankl-matching} implies the conclusion of Theorem \ref{thm:main} under the (stronger) condition $n \geq (2r+1)k-r$.

\subsection*{Our proof techniques}
Our main tool is the following `stability' version of Theorem \ref{thm:main}.
\begin{thm}
\label{thm:stability}There exists an absolute constant $C_0>0$ such that the following holds. Let $r,k \in \mathbb{N}$ with $k \geq C_0r^2$, let $s \geq C_0 \sqrt{\log k}$, let $t \in \mathbb{N}$ with $t \geq s^2 k/n$, let $n\ge2k+s\sqrt{k}$, and let $\f\subset \binom{\left[n\right]}{k}$
be a family satisfying $\mh\left(\f^{\uparrow}\right)\le1-2^{-r}$
and $\left|\f\right|\ge\binom{n}{k}-\binom{n-r}{k}-\binom{n-r-t}{k-1}$. Then there exists $R \in \binom{[n]}{r}$ such
that $|\{S \in \f:\ S \cap R = \emptyset\}| \le 2^r \exp(-\Theta(s^2 k/n)) \binom{n-r}{k}$.
\end{thm}

Here, for $\f \subset \p([n])$, we write $\f^{\uparrow} : = \{S \subset [n]:\ T \subset S \text{ for some } T \in \f\}$ for the {\em up-closure} of $\f$. For $0 < p < 1$ and $\g \subset \p([n])$, $\mu_p(\g)$ denotes the $p$-biased measure of $\g$, defined in Section \ref{sec:defns} below.

Roughly speaking, our strategy for proving Theorem \ref{thm:stability} is as follows. Instead of working with the uniform measure on ${[n] \choose k}$, we consider the up-closure $\mathcal{F}^{\uparrow}$ of our family $\mathcal{F}$, and we work with the {\em $p$-biased measure} on $\p([n])$, where $p \approx k/n$. It is well-known that $\mu_p(\f^\uparrow)$ approximately bounds $|\f|/{n \choose k}$ from above, for an appropriate choice of $p$. More precisely, we choose $p$ to be slightly larger than $k/n$, and use the lower bound on $|\mathcal{F}|$ to show that $\mu_p(\f^\uparrow) \approx 1-(1-p)^r$. Combined with the fact that $\mu_{1/2}(\f^{\uparrow}) \leq 1-2^{-r}$, this implies an upper bound on the derivative of the function $q \mapsto \mu_q(\f^{\uparrow})$, at some $q \in (p,1/2)$. But by Russo's Lemma, this derivative is precisely $I^q[\f^{\uparrow}]$, the {\em influence} of $\f^{\uparrow}$ with respect to the $q$-biased measure; we deduce that $I^q[\f^\uparrow]$ is close to its minimum possible value. We then use a recent structure theorem for families with small influence (proved in \cite{iso-stability}) to deduce that $\f^{\uparrow}$ must be close (with respect to the $q$-biased measure) to a family of the form $\{S \subset [n]: S \cap R \neq \emptyset\}$, for some $R \in \binom{[n]}{r}$. Finally, we deduce from this that $\mathcal{F}$ is almost contained in a family of the form $\{S \in \binom{[n]}{k}: S \cap R \neq \emptyset\}$. Note that a similar strategy was used to obtain the stability results in \cite{long-paper}; indeed, we use here some of the lemmas from that paper.

We deduce Theorem \ref{thm:main} from Theorem \ref{thm:stability} using a combinatorial `bootstrapping' argument, involving an analysis of cross-intersecting families.

\section{Definitions, notation and tools}
\label{sec:defns}

\subsection*{Definitions and notation}
In this paper, all logarithms are to the base 2. A {\em dictatorship} is a family of the form $\{S \subset [n]:\ j \in S\}$ or $\{S \in \binom{[n]}{k}:\ j \in S\}$ for some $j \in [n]$. For $j \in [n]$, we write $\mathcal{D}_{j}:= \{S \in \binom{[n]}{k}:\ j \in S\}$ for the corresponding dictatorship. If $R \subset [n]$, we write $\mathcal{S}_R : = \{S \subset [n]:\ R \subset S\}$, and we write $\mathrm{OR}_R := \{S\subset [n]:\ S \cap R \neq \emptyset\}$.

 A family $\mathcal{F} \subset \p([n])$ is said to be {\em increasing} (or an {\em up-set}) if it is closed under taking supersets, i.e.\ whenever $A \subset B$ and $A \in \mathcal{F}$, we have $B \in \mathcal{F}$; it is said to be {\em decreasing} (or a {\em down-set}) if it is closed under taking subsets.
 
If $\mathcal{F} \subset \p([n])$ and $l \in [n]$, we write $\mathcal{F}^{(l)} : = \{F \in \mathcal{F}:\ |F|=l\}$. Hence, for example,
$$(\mathrm{OR}_{R})^{(k)} = \{S \in \binom{[n]}{k}:\ S \cap R \neq \emptyset\}.$$

If $\mathcal{F} \subset \p([n])$, we define the {\em dual} family $\f^*$ by $\mathcal{F}^* = \{[n] \setminus A:\ A \notin \f\}$. We denote by $\mathcal{F}^{\uparrow}$ the up-closure of $\mathcal{F}$, i.e.\ the minimal increasing subfamily of $\p([n])$ which contains $\mathcal{F}$.

If $\f \subset \p([n])$ and $C \subset B \subset [n]$, we define $\f_{B}^{C}:= \left\{S\in\mathcal{P}\left(\left[n\right]\setminus B\right)\,:\, S\cup C \in \mathcal{F}\right\}$.

A family $\mathcal{F} \subset \mathcal{P}([n])$ is said to be a {\em subcube} if $\mathcal{F} = \{S \subset [n]:\ S \cap B = C\}$, for some $C \subset B \subset [n]$, and it is said to be an {\em increasing subcube} if $\mathcal{F} = \{S \subset [n]:\ B \subset S\}$, for some $B \subset [n]$.

We say a pair of families $\mathcal{A}, \mathcal{B} \subset \p([n])$ are {\em cross-intersecting} if $A \cap B \neq \emptyset$ for any $A \in \mathcal{A}$ and any $B \in \mathcal{B}$.

If $X$ is a set and $\mathcal{A} \subset X$, we write $1_{\A}$ for the {\em indicator function} of $\mathcal{A}$, i.e., the Boolean function
$$1_{\mathcal{A}}:\ X \to \{0,1\};\quad 1_{\mathcal{A}}(x) = \begin{cases} 1 & \textrm{ if } x \in \mathcal{A};\\ 0 & \textrm{ if } x \notin \mathcal{A}.\end{cases}$$

By identifying $\{0,1\}^n$ with $\p([n])$ in the usual way (identifying a vector $x \in \{0,1\}^n$ with the set $\{i:\ x_i=1\} \subset [n]$), we may identify Boolean functions on $\{0,1\}^n$ with Boolean functions on $\p([n])$, and therefore with subfamilies of $\mathcal{P}([n])$. We will sometimes write Boolean functions on $\{0,1\}^n$ using the AND ($\wedge$) and OR ($\vee$) operators. Hence, for example,
$$f: \{0,1\}^n \to \{0,1\};\quad f(x_1,\ldots,x_n) \mapsto x_1 \vee (x_2 \wedge x_3)$$
corresponds to the subfamily $\{S \subset [n]:\ 1 \in S \textrm{ or } \{2,3\} \subset S\} \subset \p([n])$.

For $p\in [0,1]$, the {\em $p$-biased measure on $\p([n])$} is defined by
$$\mu_p(S) = p^{|S|} (1-p)^{n-|S|}\quad \forall S \subset [n].$$
In other words, we choose a random set by including each $j \in [n]$ independently with probability $p$. For $\f \subset \p([n])$, we define $\mu_p(\f) = \sum_{S \in \f} \mu_p(S)$.

We remark that if $C \subset B \subset [n]$, then $\mu_p(\mathcal{F}_B^C)$ refers to the $p$-biased measure on $\p([n] \setminus B)$, not on $\p([n])$, since we regard $\f_{B}^{C}$ as a subset of $\p([n] \setminus B)$.

If $f:\p([n]) \to \{0,1\}$ is a Boolean function, we define the {\em influence of $f$ in direction $i$} (with respect to $\mu_p$) by
$$\Inf^p_{i}[f] := \mu_p(\{S \subset [n]:\ f(S) \neq f(S\Delta \{i\})\}).$$
We define the {\em total influence} of $f$ (w.r.t. $\mu_p$) by $I^p[f] := \sum_{i=1}^{n} \Inf^p_{i}[f]$.

Similarly, if $\A \subset \p([n]$, we define the {\em influence of $\A$ in direction $i$} (w.r.t. $\mu_p$) by $\Inf^p_i[\A] : = \Inf^p_i[1_{\A}]$, and we define {\em total influence} of $\A$ (w.r.t. $\mu_p$) by $I^p[\A] : = I^p[1_{\A}]$.

\subsection*{Tools}
We will use the following `biased version' of the Erd\H{o}s-Ko-Rado theorem, first obtained by Ahlswede and Katona~\cite{AK77} in 1977.
\begin{thm}
\label{thm:biased-ekr}
Let $0<p\leq1/2$. Let $\f \subset \mathcal{P}([n])$ be an intersecting family. Then $\mu_p(\f) \leq p$. If $p < 1/2$, then equality holds if and only if $\f = \{S \subset [n]:\ j \in S\}$ for some $j \in [n]$.
\end{thm}

We will use the following special case of the well-known inequality of Harris \cite{harris} (which is itself a special case of the FKG inequality \cite{fkg}).
\begin{lem}[Harris]
\label{lem:harris}
Let $0 < p < 1$. Then for any increasing sets $\mathcal{A},\mathcal{B} \subset \mathcal{P}([n])$, $\mu_p(\mathcal{A} \cap \mathcal{B}) \geq \mu_p(\mathcal{A})\mu_p(\mathcal{B})$. The same inequality holds if $\mathcal{A}$ and $\mathcal{B}$ are decreasing.
\end{lem}

By repeatedly applying Lemma \ref{lem:harris}, one immediately obtains the following well-known corollary.
\begin{cor}
\label{cor:harris}
Let $r \in \mathbb{N}$, let $0 < p < 1$, and suppose $\mathcal{A}_1,\ldots,\mathcal{A}_r \subset \p([n])$ are increasing. Then
$$\mu_p(\A_1 \cap \ldots \cap \A_r) \geq \prod_{i=1}^{r} \mu_p(\A_i).$$
The same inequality holds if $\A_1,\ldots,\A_r$ are decreasing.
\end{cor}

The following `biased isoperimetric inequality' is well-known; it appears for example in \cite{Kahn-Kalai}.
\begin{thm}
\label{thm:skewed-iso}
If $0 < p < 1$ and $\mathcal{A} \subset \p([n])$ is increasing, then
\begin{equation} \label{eq:skewed-iso} pI^p[\A] \geq \mu_p(\A) \log_p (\mu_p(\A)).\end{equation}
\end{thm}

We will need the following `stability' version of Theorem \ref{thm:skewed-iso}, proved by Keller and the authors in \cite{iso-stability}.
\begin{thm}
\label{thm:iso-stability}
For each $\eta>0$, there exist $C_1 = C_1(\eta),\ c_0 = c_0(\eta)>0$ such that the following holds. Let $0<p\leq 1-\eta$, and let $0 \leq \epsilon \leq c_{0}/\ln(1/p)$. Let $\mathcal{A} \subset \p([n])$ be an increasing family such that
$$pI^p[\A] \leq \mu_{p}(\A) \left(\log_{p}\left(\mu_{p}(\A) \right)+\epsilon\right).$$
Then there exists an increasing subcube $\mathcal{C} \subset \p([n])$ such that
$$\mu_p(\A \Delta \mathcal{C}) \leq \frac{C_1\epsilon\ln(1/p)}{\ln\left(\frac{1}{\epsilon\ln(1/p)}\right)}\mu_p(\A).$$
\end{thm}

We will need the well-known lemma of Russo \cite{Russo}, which relates the derivative of the function $p \mapsto \mu_p(A)$ to the total influence $I^p(A)$, where $A \subset \{0,1\}^n$ is increasing.
\begin{lem}[Russo's lemma]\label{lem:russo}
Let $\A \subset \p([n])$ be increasing, and let $0 < p_0 < 1$. Then
$$\frac{d \mu_p(\A)}{dp} \Big|_{p=p_0} = I^{p_0}[\A].$$
\end{lem}

We need the following lemma from \cite{long-paper}, which follows from Russo's lemma and Theorem \ref{thm:skewed-iso}.
\begin{lem}
\label{lem:mono}
If $\mathcal{A} \subset \p([n])$ is increasing, then the function $p \mapsto \log_p(\mu_p(\A))$ is monotone non-increasing on $(0,1)$.
\end{lem}

We will also need the following Chernoff bound.

\begin{lem}
\label{lem:chernoff-small}
Let $n \in \mathbb{N}$, let $0 < \delta,p <1$ and let $X \sim \textrm{Bin}(n,p)$. Then
\begin{equation}\label{eq:chernoff-small} \Pr[X \leq (1-\delta)np] < e^{-\delta^2np/2}.\end{equation}
\end{lem}

The following lemma (combined with the Chernoff bound (\ref{eq:chernoff-small})) will allow us to bound $|\g|/{n \choose k}$ from above in terms of $\mu_p(\g^{\uparrow})$, where $\g \subset \binom{[n]}{k}$ and $p$ is slightly larger than $k/n$.
\begin{lem}
\label{lem:going-up}
Let $k,n \in \mathbb{N}$, let $0 < \alpha,p < 1$ and let $\g \subset \binom{\left[n\right]}{k}$ be a family with $\left|\g\right|=\alpha\binom{n}{k}$.
Then 
\[
\mu_{p}\left(\g^{\uparrow}\right)\ge\alpha\Pr\left[\mathrm{Bin}\left(n,p\right) \geq k\right].
\]
 \end{lem}
\begin{proof}
For each $l\ge k$, the local LYM inequality (see e.g.\ \cite[\S 5]{bollobas-book}) implies that $|(\g^{\uparrow})^{\left(l\right)}|/{n \choose l} \ge |\mathcal{G}|/{n \choose k} = \alpha$.
Hence,
\begin{align*}
\mu_{p}\left(\g^{\uparrow}\right) & \ge\sum_{l=k}^{n}p^{l}\left(1-p\right)^{n-l}\alpha\binom{n}{l}\\
 & =\alpha\Pr\left[\mathrm{Bin}\left(n,p\right) \geq k\right],
\end{align*}
as required.
\end{proof}

Finally, we need the following immediate consequence of a lemma of Hilton (see \cite{Frankl87}).
\begin{lem}
\label{lem:hilton}
Let $n,k,l,t \in \mathbb{N}$ with $k+l \leq n$. Let $\mathcal{A} \subset \binom{[n]}{k},\ \mathcal{B} \subset \binom{[n]}{l}$ be cross-intersecting families. If $|\mathcal{A}| \geq {n \choose k} - {n-t \choose k}$, then $|\mathcal{B}| \leq {n-t \choose l-t}$.
\end{lem}

\section{Proofs of the main results}

Our first aim is to prove Theorem \ref{thm:stability}; for this, we need some preliminary lemmas.
\begin{lem}
\label{lem:Chernoff} Let $s >0$ and let $t \in \mathbb{N}$ with $t \geq s^2 k/n$. Let $n,k \in \mathbb{N}$ with $n\ge2k+s\sqrt{k}$, and let $p = \tfrac{k/n+0.5}{2}$. If $\f\subset \binom{\left[n\right]}{k}$ with $\left|\f\right|\ge\binom{n}{k}-\binom{n-r}{k}-\binom{n-r-t}{k-1}$, then
\[
\mu_{p}\left(\f^{\uparrow}\right)\ge1-(1-p)^r-\exp(-\Omega(s^2 k/n)).
\]
\end{lem}
\begin{proof}
The Kruskal-Katona Theorem implies that 
\begin{align*}
\left(\f^{\uparrow}\right)^{\left(l\right)} & \ge\binom{n}{l}-\binom{n-r}{l}-\binom{n-r-t}{l-1}\\
 & =|\left(x_{1}\vee x_2 \vee \ldots \vee x_{r-1} \vee \left(x_{r}\wedge\left(x_{r+1}\vee x_{r+2}\vee\cdots\vee x_{r+t}\right)\right)\right)^{\left(l\right)}|
\end{align*}
 for any $l\ge k$. It follows that 
\begin{align*}
\mu_{p}\left(\f^{\uparrow}\right) & \ge\mu_{p}\left(x_{1}\vee x_2 \vee \ldots \vee x_{r-1} \vee \left(x_{r}\wedge\left(x_{r+1}\vee x_{r+2}\vee\cdots\vee x_{r+t}\right)\right)\right)\\
&-\Pr\left[\mathrm{Bin}\left(n,p\right)<k\right]\\
 & =1-(1-p)^r-p(1-p)^{r+t-1}-\Pr\left[\mathrm{Bin}\left(n,p\right)<k\right].
\end{align*}
 The Chernoff bound in Lemma \ref{lem:chernoff-small} (applied with $\delta = 1-k/(np) = \Omega(s\sqrt{k}/n)$), together with our condition on $t$, completes the proof. \end{proof}
\begin{lem}
\label{lem:russo}
Let $r,n \in \mathbb{N}$, let $0 < p < 1/2$ and let $0 < \eta < 1$. If $\mathcal{A} \subset \p([n])$ is increasing with $\mu_{1/2}(\mathcal{A}) \leq 1-2^{-r}$ and
$$\mu_{p}\left(\mathcal{A}\right)\ge1-(1-p)^r-\eta,$$
then there exists $p'\in\left(p,\frac{1}{2}\right)$
such that
$$I^{p'}\left[\mathcal{A}\right]\le I^{p'}\left[x_{1}\vee \ldots \vee x_{r}\right]+\frac{\eta}{0.5-p}.$$
\end{lem}
\begin{proof}
By Russo's lemma (Lemma \ref{lem:russo} above), we have
\[
\intop_{p}^{0.5}I^{q}\left[\mathcal{A}\right]dq=\mh\left(\mathcal{A}\right)-\mu_{p}\left(\mathcal{A}\right)\le1-2^{-r}-\left(1-(1-p)^r\right)+\eta.
\]
Hence,
\[
\intop_{p}^{0.5}\left(I^{q}\left[\mathcal{A}\right]-I^{q}\left[x_{1}\vee \ldots \vee x_{r}\right]\right)dq\le \eta.
\]
This implies that for some $p'\in (p,0.5)$ we have
$$I^{p'}\left[\mathcal{A}\right] - I^{p'}\left[x_{1}\vee \ldots \vee x_{r}\right] \leq \frac{\eta}{0.5-p},$$
as required.
\end{proof}
\begin{lem}
\label{lem:stab-iso}
There exist absolute constants $\delta_0,\epsilon_0, C_2 >0$ such that the following holds. Let $0 \leq \delta < \delta_0$, $0 \leq \epsilon < \epsilon_0$ and $1/4 \leq p < p' < 1/2$. If $\mathcal{A} \subset \p([n])$ is increasing with $\mu_{1/2}(\mathcal{A}) \leq 1-2^{-r}$, $\mu_{p}(\A) \geq 1-(1-p)^r(1+\delta)$ and
$$I^{p'}\left[\mathcal{A}\right] - I^{p'}\left[x_{1}\vee \ldots \vee x_{r}\right] \leq \epsilon (1-p')^r,$$
then there exists $R \in \binom{[n]}{r}$ such that
$$\mu_{p'}(\mathcal{A}_{R}^{\varnothing}) \leq C_2(\epsilon+\delta).$$
\end{lem}
\begin{proof}
Note that for any family $\B \subset \p([n])$, we have $I^{p'}[\B] = I^{1-p'}[\B^*]$. Hence, by hypothesis, we have
$$I^{1-p'}\left[\mathcal{A}^*\right] - r(1-p')^{r-1} = I^{1-p'}\left[\mathcal{A}^*\right] - I^{1-p'}\left[x_{1}\wedge \ldots \wedge x_{r}\right] \leq \epsilon (1-p')^r.$$
Since $\mathcal{A}^*$ is increasing and $\mu_{1/2}(\A^*) = 1-\mu_{1/2}(\A) \geq 2^{-r}$, by Lemma \ref{lem:mono} we have $\mu_{1-p'}(\A^*) \geq (\mu_{1/2}(\A^*))^{\log_{1/2}(1-p')} \geq (1-p')^r$. Similarly, since $\mu_{1-p}(\A^*) = 1-\mu_p(\A^*) \leq (1-p)^r(1+\delta)$, we have
$$\mu_{1-p'}(\A^*) \leq (\mu_{1-p}(\A^*))^{\log_{1-p}(1-p')} \leq ((1-p)^r(1 + \delta))^{\log_{1-p}(1-p')} \leq (1-p')^r(1 + 3\delta),$$ 
provided $\delta_0$ is sufficiently small. Therefore,
$$\mu_{1-p'}(\A^*) \log_{1-p'}(\mu_{1-p'}(\A^*)) \geq (1-p')^r \log_{1-p'}((1-p')^r(1 + 3\delta)) \geq (r-11\delta)(1-p')^r.$$
It follows that
$$(1-p')I^{1-p'}\left[\mathcal{A}^*\right] - \mu_{1-p'}(\A^*) \log_{1-p'}(\mu_{1-p'}(\A^*)) \leq (\epsilon +11 \delta)\mu_{1-p'}(\A^*).$$

Applying Theorem \ref{thm:iso-stability} (with $\eta = 1/4$, with $1-p'$ in place of $p$ and with $\epsilon + 11\delta$ in place of $\epsilon$) to the family $\mathcal{A}^*$, we see that there exists $R \subset [n]$ such that
\begin{equation}\label{eq:diffbound} \mu_{1-p'}(\A^* \Delta \mathcal{S}_{R}) \leq C_2 (\epsilon + \delta) (1-p')^r,\end{equation}
where $C_2>0$ is an absolute constant, provided $\epsilon_0,\delta_0$ are sufficiently small. We claim that $|R| = r$. Indeed, if $|R| > r$, then
$$\mu_{1-p'}(\A^* \Delta \mathcal{S}_R) \geq \mu_{1-p'}(\A^*) - \mu_{1-p'}(\mathcal{S}_R) \geq (1-p')^r - (1-p')^{r+1} = p'(1-p')^r,$$
contradicting (\ref{eq:diffbound}) provided $\epsilon_0,\delta_0$ are sufficiently small. Similarly, if $|R| < r$, then
\begin{align*} \mu_{1-p'}(\A^* \Delta \mathcal{S}_R) & \geq \mu_{1-p'}(\mathcal{S}_R) - \mu_{1-p'}(\A^*)\\
&  \geq (1-p')^{r-1} - (1+3\delta)(1-p')^r\\
& = (1-p')^{r-1} (p' - 3(1-p')\delta),
\end{align*}
again contradicting (\ref{eq:diffbound}) provided $\epsilon_0,\delta_0$ are sufficiently small. This proves the claim. It follows that
\begin{align*} \mu_{p'}(\A_{R}^{\emptyset})& = (1-p')^{-r} \mu_{p'}(\A \setminus \mathrm{OR}_R)\\
&  \leq (1-p')^{-r} \mu_{p'}(\A \Delta \mathrm{OR}_{R})\\
& = (1-p')^{-r}\mu_{1-p'}(\A^* \Delta \mathcal{S}_R)\\
& \leq C_2(\epsilon + \delta),
\end{align*}
as required.
\end{proof}

\begin{proof}[Proof of Theorem \ref{thm:stability}.]
Let $n,k,r,s$ and $t$ be as in the statement of the theorem, where $C_0$ is to be chosen later. Let $\f\subset \binom{\left[n\right]}{k}$
be a family satisfying $\mh\left(\f^{\uparrow}\right)\le1-2^{-r}$
and $\left|\f\right|\ge\binom{n}{k}-\binom{n-r}{k}-\binom{n-r-t}{k-1}$.

Let $p = \tfrac{k/n+0.5}{2}$. By Lemma \ref{lem:Chernoff}, we have
$$\mu_{p}\left(\f^{\uparrow}\right)\ge1-(1-p)^r-\exp(-\Theta(s^2 k/n)).$$
Applying Lemma \ref{lem:russo} with $\eta = \exp(-\Theta(s^2 k/n))$ and $\mathcal{A} = \f^{\uparrow}$, yields $p' \in (p,\tfrac{1}{2})$ such that
$$I^{p'}\left[\f^{\uparrow} \right]\le I^{p'}\left[x_{1}\vee \ldots \vee x_{r}\right]+\frac{\exp(-\Theta(s^2 k/n))}{0.5-p}.$$
Provided $C_0$ is sufficiently large, we may apply Lemma \ref{lem:stab-iso} with $\delta = 2^r \exp(-\Theta(s^2 k/n))$ and 
$$\epsilon = \frac{\exp(-\Theta(s^2 k/n))}{(0.5-p)(1-p')^{r}} \leq \frac{2^r \sqrt{k}}{s} \exp(-\Theta(s^2 k/n)) \leq 2^r\exp(-\Theta(s^2 k/n)),$$
yielding
\begin{equation}\label{eq:boundy1} \mu_{p'}((\f^{\uparrow})_{R}^{\varnothing}) \leq 2^r \exp(-\Theta(s^2 k/n))\end{equation}
for some $p' \in (p,1/2)$ and some $R \in \binom{[n]}{r}$.

Applying Lemma \ref{lem:going-up} with $\g = (\f^{\uparrow})_{R}^{\varnothing}$, with $n-r$ in place of $n$, and with $p'$ in place of $p$, we obtain
\begin{equation}\label{eq:boundy2} \frac{|\f_{R}^{\varnothing}|}{\binom{n-r}{k}} \leq \frac{\mu_{p'}((\f^{\uparrow})_{R}^{\varnothing})}{\Pr[\Bin(n-r,p') \geq k]}.\end{equation}
Applying the Chernoff bound in Lemma \ref{lem:chernoff-small} with $n-r$ in place of $n$, and with $\delta := 1-k/(p'(n-r)) = \Omega(s \sqrt{k}/n)$, we obtain
\begin{equation}\label{eq:boundy3} \Pr[\Bin(n-r,p') \geq k] > 1-\exp(-\Theta(s^2 k/n)).\end{equation}
Combining (\ref{eq:boundy1}), (\ref{eq:boundy2}) and (\ref{eq:boundy3}), we obtain
$$\frac{|\f_{R}^{\varnothing}|}{{n - r \choose k}} < \frac{\mu_{p'}((\f^{\uparrow})_{R}^{\varnothing})}{1-\exp(-\Theta(s^2 k/n))} \leq 2^r \exp(-\Theta(s^2 k/n)),$$
completing the proof of Theorem \ref{thm:stability}.
\end{proof}

Before proving Theorem \ref{thm:main}, we need some additional lemmas.

\subsection*{The FKG bound}
We need the following well-known upper bound on the $p$-biased measure of the union of $r$ $1$-intersecting subfamilies of $\p([n])$; we provide a proof for completeness.
\begin{lem}
\label{lem:fkg}
If $\f_1,\ldots,\f_r \subset \p([n])$ are intersecting families, and $0 < p \leq 1/2$, then
$$\mu_p(\f_1 \cup \ldots \cup \f_r) \leq 1 - (1-p)^r.$$
\end{lem}
\begin{proof}
By replacing $\mathcal{F}_i$ with $\mathcal{F}_i^{\uparrow}$ for each $i$, if necessary, we may assume that each $\mathcal{F}_i$ is increasing. For each $i$, since $\f_i$ is intersecting, Theorem \ref{thm:biased-ekr} implies that $\mu_p(\f_i) \leq p$, and therefore $\mu_{p}(\f_i^c) \geq 1-p$. Hence, using Corollary \ref{cor:harris} (applied to the down-sets $\f_1^c,\ldots,\f_r^c$), we have
$$\mu_{p}(\f_1 \cup \ldots \cup \f_r) = 1-\mu_{p}(\f_1^c \cap \ldots \cap \f_r^c) \leq 1-\prod_{i=1}^{r} \mu_{p} (\f_i^c) \leq 1-(1-p)^{r},$$
as required.
\end{proof}
Clearly, Lemma \ref{lem:fkg} is sharp, as can be seen by taking $\f_i = \{S \subset [n]:\ i \in S\}$ for each $i \in [r]$.

\subsection*{Upper bounds on linear combinations of sizes of cross-intersecting families}

\begin{lem}
\label{lem:cross} For each constant $C_{1}>0$,
there exists a constant $C_{2} = C_2\left(C_{1}\right)>0$ such that the following
holds. Let $\frac{n}{C_{1}}<k_{1}<\frac{n}{2}-C_{2},\ \frac{n}{C_{1}}<k_{2}<\frac{n}{2}-C_{2}$
with $\left|k_{1}-k_{2}\right|\le C_{1}$, and let $t_{0} \in \mathbb{N}$ with $t_0 \ge C_{2}/\log\left(\frac{n-k_{1}}{k_{1}}\right)$.
Suppose that $\g_{1}\subset \binom{\left[n\right]}{k_{1}},\ \g_{2}\subset \binom{\left[n\right]}{k_{2}}$
are cross-intersecting families with $\left|\g_{1}\right|\le\binom{n-t_{0}}{k_{1}-t_{0}}$.
Then 
\[
\left|\g_{2}\right|+C_{1}\left|\g_{1}\right|\le\binom{n}{k_{2}},
\]
and equality holds only if $\g_1 = \emptyset$.
\end{lem}
\begin{proof}
Choose $t \in \mathbb{N}$ such that $\binom{n-t-1}{k_1-t-1}\le\left|\g_{1}\right|\le\binom{n-t}{k_{1}-t}$. Note that $t\ge t_{0}\ge C_{2} / \log\left(\frac{n-k_{1}}{k_{1}}\right)$.
By Lemma \ref{lem:hilton}, we have $\left|\g_{2}\right|\le\binom{n}{k_{2}}-\binom{n-t-1}{k_{2}}$.
So it suffices to prove that $\frac{\binom{n-t-1}{k_{2}}}{\binom{n-t}{k_{1}-t}} > C_{1}$.

Observe that
\[
\frac{\binom{n-t-1}{k_{2}}}{\binom{n-t}{k_{1}-t}}=\Theta_{C_{1}}\left(\frac{\binom{n-t}{k_{1}}}{\binom{n-t}{k_{1}-t}}\right),
\]
and 
\[
\frac{\binom{n-t}{k_{1}}}{\binom{n-t}{k_{1}-t}}=\frac{\left(n-k_{1}\right)\cdot\left(n-k_{1}-1\right)\cdot \ldots \cdot (n-k_{1}-t+1)}{(k_{1})\cdot(k_1-1) \cdot \ldots\cdot\left(k_{1}-t+1\right)}\ge\left(\frac{n-k_{1}}{k_{1}}\right)^{t}\ge2^{C_{2}}.
\]
Hence, 
\[
\frac{\binom{n-t-1}{k_{2}}}{\binom{n-t}{k_{1}-t}} = \Theta_{C_{1}}\left(\frac{\binom{n-t}{k_{1}}}{\binom{n-t}{k_{1}-t}}\right) = \Omega_{C_1}(2^{C_2}) > C_{1},
\]
 provided $C_{2}$ is sufficiently large depending on $C_1$, as required.
\end{proof}

\subsection*{Approximate containment in dictatorships}
We now show that if $\f = \f_1 \cup \ldots \cup \f_r$ with $\f_i \subset \binom{[n]}{k}$ an intersecting family for each $i \in [r]$, and $|\f| \approx {n \choose k} - {n -r \choose k}$, then not only is $\f$ well-approximated by $(\mathrm{OR}_{R})^{(k)}$ for some $R \in \binom{[n]}{r}$, but in fact each $\f_{i}$ is well-approximated by a (different) dictatorship $\d_{j}$ (with $j \in R$). Specifically, we prove the following. 
\begin{lem}
\label{lem:stab-fi}
There exists an absolute constant $C_0>0$ such that the following holds. Let $r,k \in \mathbb{N}$ with $k \geq C_0r^2$, let $s \geq C_0 \sqrt{\log k}$, let $t \in \mathbb{N}$ with $t \geq s^2 k/n$, let $n\ge2k+s\sqrt{k}$, and let $\f = \f_1 \cup \ldots \cup \f_r$, where $\f_i \subset \binom{[n]}{k}$ is an intersecting family for each $i \in [r]$. If $\left|\f\right|\ge\binom{n}{k}-\binom{n-r}{k} -\binom{n-r-t}{k-1}$, then there exists a set $R \in \binom{[n]}{r}$ and a permutation $\pi\in \mathrm{Sym}(R)$
such that $\left|\left(\f_{i}\right)_{\{\pi\left(i\right)\}}^{\varnothing}\right|\le2^{2r}e^{-\Theta\left(s^{2}k/n\right)}\binom{n-1}{k}$ for each $i \in R$. 
\end{lem}
\begin{proof}
First note that by Theorem \ref{thm:stability}, we have $\left|\f_{R}^{\varnothing}\right|\le2^{r}e^{-\Theta\left(s^{2} k/n\right)}\binom{n-r}{k}$ for some $R \in \binom{[n]}{r}$; without loss of generality, we may assume that $R = [r]$. Hence, 
\begin{align*}
\left|\f_{\left[r\right]}^{\left\{ j\right\} }\right| & \ge\binom{n-r}{k-1}\left(1-2^{r}e^{-\Theta\left(s^{2} k/n\right)}\frac{n-r-k+1}{k}\right)\\
 & =\binom{n-r}{k-1}\left(1-2^{r}e^{-\Theta\left(s^{2} k/n\right)}\right)
\end{align*}
 for each $j \in [r]$.
 
 Note that for each $j_{1} \neq j_{2}\in\left[r\right]$, the families $\left(\f_{i}\right)_{\left[r\right]}^{\left\{ j_{1}\right\} },\left(\f_{i}\right)_{\left[r\right]}^{\left\{ j_{2}\right\} }$
are cross-intersecting. So we may assume, without loss of generality, that
$\mh\left(\left(\left(\f_{i}\right)_{\left[r\right]}^{\{j\}}\right)^{\uparrow}\right)\le \tfrac{1}{2}$
for any $j\ne i$. 

Fix $j \in [r]$. By Lemma \ref{lem:going-up} together with the
Chernoff bound (\ref{eq:chernoff-small}), we have $\mu_{\frac{1}{2}}\left(\left(\f_{\left[r\right]}^{\left\{ j\right\} }\right)^{\uparrow}\right)\ge1-2^{r}e^{-\Theta\left(s^{2} k/n\right)}$.
Using Corollary \ref{cor:harris}, we have
\[
1-\mu_{\frac{1}{2}}\left(\left(\f_{\left[r\right]}^{\left\{ j\right\} }\right)^{\uparrow}\right)\ge\prod_{i=1}^{r} \left(1-\mu_{\frac{1}{2}}\left(\left(\left(\f_{i}\right)_{\left[r\right]}^{\left\{ j\right\} }\right)^{\uparrow}\right)\right)\ge\left(\frac{1}{2}\right)^{r-1}\left(1-\mh\left(\left((\f_j)_{\left[r\right]}^{\left\{ j\right\} }\right)^{\uparrow}\right)\right).
\]
 Rearranging, we obtain
\[
\mh\left(\left(\left(\f_{j}\right)_{\left\{ j\right\} }^{\left\{ j\right\} }\right)^{\uparrow}\right)\ge\mh\left(\left(\left(\f_{j}\right)_{\left[r\right]}^{\left\{ j\right\} }\right)^{\uparrow}\right)\ge1-2^{2r}e^{-\Theta\left(s^{2} k/n\right)}.
\]
 Hence $\mh\left(\left(\left(\f_{j}\right)_{\left\{ j\right\} }^{\varnothing}\right)^{\uparrow}\right)\le2^{2r}e^{-\Theta\left(s^{2} k/n\right)}$
and the lemma follows from Lemma \ref{lem:going-up} and the Chernoff bound (\ref{eq:chernoff-small}).
\end{proof}

Finally, we need the following easy combinatorial inequality.
\begin{claim}
\label{claim:indicator}
Let $\f_1,\ldots,\f_r \subset \binom{[n]}{k}$ and let $\f = \f_1 \cup \ldots\cup \f_r$. Then
 \[
\left|\f\right|\le\binom{n}{k}-\binom{n-r}{k}+\sum_{j=1}^{r} \left(\left|\left(\f_{j}\right)_{\left\{ j\right\} }^{\varnothing}\right|-\left|\left(\left(\f_{j}\right)_{\left[r\right]}^{\left\{ j\right\} }\right)^c\right|\right),
\]
where $\left(\left(\f_{j}\right)_{\left[r\right]}^{\{j\}}\right)^c : = \binom{[n] \setminus [r]}{k-1} \setminus \left(\f_{j}\right)_{\left[r\right]}^{\{j\}}$.
\end{claim}
\begin{proof}
It suffices to prove that
\begin{equation}\label{eq:pointwise} 1_{\f}(S) \leq 1_{\OR_{[r]}}(S) + \sum_{j=1}^{r} \left(1_{(\f_{j})_{\{j\}}^{\varnothing}}(S)-1_{\binom{[n] \setminus [r]}{k-1} \setminus (\f_{j})_{[r]}^{\{j\}}}(S \setminus \{j\})\right)\end{equation}
for all $S \in \binom{[n]}{k}$. (The statement of the claim then follows by summing (\ref{eq:pointwise}) over all $S \in \binom{[n]}{k}$.) To prove (\ref{eq:pointwise}), observe that for any set $S \in \binom{[n]}{k}$, we have
$$1_{\binom{[n] \setminus [r]}{k-1} \setminus (\f_{j})_{[r]}^{\{j\}}}(S \setminus \{j\}) = 1 \Rightarrow S \cap [r]=\{j\} \Rightarrow 1_{\OR_{[r]}}(S) = 1,$$
so the right-hand side of (\ref{eq:pointwise}) is always non-negative. Hence, we may assume that $S \in \f$. Without loss of generality, we may assume that $S \in \f_1$. If $|S \cap [r]| \geq 2$ or $S \cap [r]=\{1\}$, then 
$$1_{\binom{[n] \setminus [r]}{k-1} \setminus (\f_{j})_{[r]}^{\{j\}}}(S \setminus \{j\}) = 0 \ \forall j \in [r],\quad 1_{\OR_{[r]}}(S) = 1,$$
so the right-hand side of (\ref{eq:pointwise}) is at least 1, and we are done. If $S \cap [r] = \emptyset$, then
$$1_{\binom{[n] \setminus [r]}{k-1} \setminus (\f_{j})_{[r]}^{\{j\}}}(S \setminus \{j\}) = 0 \ \forall j \in [r],\quad 1_{(\f_{1})_{\{1\}}^{\varnothing}}(S) = 1,$$
so the right-hand side of (\ref{eq:pointwise}) is at least 1, and we are done. Finally, if $S \cap [r] = \{i\}$ for some $i > 1$, then we have
$$1_{\binom{[n] \setminus [r]}{k-1} \setminus (\f_{j})_{[r]}^{\{j\}}}(S \setminus \{j\}) = 1$$
only if $j=i$, whereas $1_{\OR_{[r]}}(S) = 1_{(\f_{1})_{\{1\}}^{\varnothing}}(S) = 1$, so the right-hand side of (\ref{eq:pointwise}) is at least 1, and we are done.
\end{proof}

\begin{proof}[Proof of Theorem \ref{thm:main}.]
Let $\f = \f_1 \cup \ldots \cup \f_i$, where $\f_i \subset \binom{\left[n\right]}{k}$ is an intersecting family for each $i \in [r]$, and suppose that $|\f| \geq {n \choose k} - {n-r \choose k}$. Then $\mathcal{F}$ cannot contain $r+1$ pairwise disjoint sets, so by Theorem \ref{thm:frankl-matching}, if $n \geq (2r+1)k-r$, we have $|\f| \leq {n \choose k} - {n-r \choose k}$, with equality only if $\mathcal{F} = (\mathrm{OR}_R)^{(k)}$ for some $R \in \binom{[n]}{r}$. Hence, we may assume throughout that $n \leq (2r+1)k-r-1$. Moreover, by choosing $C = C(r)$ to be sufficiently large, we may assume throughout that $n \geq n_0(r)$ for any $n_0(r) \in \mathbb{N}$.

By Theorem \ref{thm:stability} (applied with $s = C(r)k^{1/6}$, where $C(r) \in \mathbb{N}$ is to be chosen later), Lemma \ref{lem:fkg} and Lemma \ref{lem:stab-fi}, there exists a set $R \in \binom{[n]}{r}$ and a permutation $\pi \in \mathrm{Sym}(R)$ such that
$$\left|\left(\f_{i}\right)_{\{\pi\left(i\right)\}}^{\varnothing}\right|\le2^{2r}e^{-\Theta\left(s^{2}k/n\right)}\binom{n-1}{k}$$
for each $i \in R$. Without loss of generality, we may assume that $R = [r]$ and $\pi = \textrm{Id}$, so that
\begin{equation}\label{eq:unif-bound} \left|\left(\f_{i}\right)_{\{i\}}^{\varnothing}\right|\le2^{2r}e^{-\Theta\left(s^{2} k/n\right)}\binom{n-1}{k}\end{equation}
for all $i \in [r]$.

By Claim \ref{claim:indicator}, we have 
\begin{equation}
\label{eq:lin-eq}
\left|\f\right|\le\binom{n}{k}-\binom{n-r}{k}+\sum_{j=1}^{r} \left(\left|\left(\f_{j}\right)_{\left\{ j\right\} }^{\varnothing}\right|-\left|\left(\left(\f_{j}\right)_{\left[r\right]}^{\left\{ j\right\} }\right)^c\right|\right).
\end{equation}

We now wish to apply Lemma \ref{lem:cross}. To this end, define
$$t_0 :=\left\lceil C_2(\max\{2^{r-1},2r+1\}) /\log\left(\frac{n-r-k}{k}\right)\right \rceil,$$
where $C_2(\cdot)$ is the function defined in Lemma \ref{lem:cross}. Since $n \geq 2k+C(r)k^{2/3} \geq 2k+k^{2/3}$, and since by assumption $ (2r+1)k \geq n \geq n_0(r)$, we have $t_0 = O_r(k^{1/3})$, and therefore
\begin{align*} \frac{\binom{n-1-(r+t_0-1)}{k-(r+t_0-1)}}{\binom{n-1}{k}} & \geq \left(\frac{k-r-t_0+2}{n-r-t_0+1}\right)^{r+t_0-1}\\
& \geq \left(\frac{1}{2r+2}\right)^{r+t_0-1}\\
& \geq \exp(-\Theta_r(k^{1/3}))\\
& > 2^{2r} \exp(-\Theta(s^{2} k/n)),
\end{align*}
provided $C = C(r) \in \mathbb{N}$ and $n_0=n_0(r) \in \mathbb{N}$ are chosen to be sufficiently large. Therefore, using (\ref{eq:unif-bound}), for all $j \in [r]$ and for all $T \subset [r] \setminus \{j\}$, we have
\begin{align*} \left|(\f_j)_{[r]}^{T} \right| &\leq \left|\left(\f_{j}\right)_{\{j\}}^{\varnothing}\right|\\
& \leq 2^{2r}e^{-\Theta\left(s^{2}k/n\right)}\binom{n-1}{k} \\
&< \binom{n-1-(r+t_0-1)}{k-(r+t_0-1)}\\
& \leq \binom{n-r-t_{0}}{k-|T|-t_{0}}.
\end{align*}
By our choice of $t_0$, we have
$$t_0 \geq C_2(\max\{2^{r-1},2r+1\}) /\log\left(\frac{n-r-k+|T|}{k - |T|}\right)$$
for all $j \in [r]$ and all $T \subset [r] \setminus \{j\}$. Hence, for any such $j$ and $T$, we may apply Lemma \ref{lem:cross} to the pair of cross-intersecting families $\mathcal{G}_1 = (\f_j)_{[r]}^{T}$ and $\mathcal{G}_2 = \left(\f_{j}\right)_{\left[r\right]}^{\left\{ j\right\}}$, with $n-r$ in place of $n$, $C_1 = \max\{2^{r-1},2r+1\}$, $k-r+1 \leq k_1 \leq k$, $k_2 = k-1$, and the above value of $t_0$. This yields
\begin{equation} \label{eq:bound} \left|\left(\f_{j}\right)_{\left\{ j\right\} }^{\varnothing}\right|-\left|\left(\left(\f_{j}\right)_{\left[r\right]}^{\left\{ j\right\} }\right)^c\right| \le 2^{r-1} \max_{T \subset [r] \setminus \{j\}} \left|\left(\f_{j}\right)_{[r]}^{T}\right|-\left|\left(\left(\f_{j}\right)_{\left[r\right]}^{\left\{ j\right\} }\right)^c\right|\leq 0\end{equation}
for each $j \in [r]$.

Combining (\ref{eq:lin-eq}) and (\ref{eq:bound}) yields $|\f| \leq {n \choose k}-{n-r \choose k}$. By hypothesis, we have $|\f| \geq {n \choose k} - {n-r \choose k}$, and therefore $|\f| = {n \choose k} - {n-r \choose k}$, so equality holds in (\ref{eq:bound}) for each $j \in [r]$. Therefore, by Lemma \ref{lem:cross}, $\left(\f_{j}\right)_{[r]}^{T} = \emptyset$ for all $T \subset [r] \setminus \{j\}$, i.e.\ $\left(\f_{j}\right)_{\left\{ j\right\}}^{\varnothing} = \emptyset$, so $\f_j \subset \d_j$ for all $j \in [r]$. Hence, $\f \subset \mathrm{OR}_{[r]}$, so $\f = (\mathrm{OR}_{[r]})^{(k)}$, as required.
\end{proof}

\subsection*{Acknowledgements}
We would like to thank Nicholas Day for informing us of the open problem considered here, and an anonymous referee for their careful reading of the paper and their helpful suggestions.

\end{document}